\documentclass[a4paper, 11pt]{article}
\usepackage{amssymb}
\usepackage{t1enc}
\usepackage[english]{babel}
\usepackage{times}
\usepackage{a4wide}
\usepackage{color}
\usepackage{amsmath, amsthm}
\usepackage{comment}
\newcounter{licznik}[section]

 \newtheorem{theorem}[licznik]{Theorem}
 \newtheorem{lemma}[licznik]{Lemma}
\newtheorem{prop}[licznik]{Proposition}
\newtheorem{cor}[licznik]{Corollary}
 \newtheorem{remark}[licznik]{Remark}

\def\tl{\tilde}

\def\ve{\varepsilon}

\def\ef{\mathcal{F}}
\def\ee{\mathbb{E}}
\def\er{\mathbb{R}}
\def\bse{\mathcal{E}}
\def\prob{\mathbb{P}}
\newcommand\ind[1]{1_{#1}}

\makeatletter
\newcommand\assumptionlabel[1]{\hspace\labelsep
                               \normalfont\bfseries #1\ \ \gdef\@currentlabel{#1}}
\newenvironment{assumption}
               {\medskip\list{}{\labelwidth\z@ \itemindent-\leftmargin
                        }}
               {\endlist}
\makeatother

\begin{document}
% Comment out for SIAM format
\title{Undiscounted optimal stopping with unbounded rewards\footnote{Research of both authors has been partly supported by NCN grant DEC-2012/07/B/ST1/03298.}}
\author{Jan Palczewski\footnote{School of Mathematics, University of Leeds, LS2 9JT, Leeds, United Kingdom} \and \L{}ukasz Stettner\footnote{Institute of Mathematics Polish Acad. Sci., Sniadeckich 8, 00-656 Warsaw, Poland, and Vistula University}}
\maketitle

\begin{abstract}
We study optimal stopping of Feller-Markov processes to maximise an undiscounted functional consisting of running and terminal rewards. In a finite-time horizon setting, we extend classical results to unbounded rewards. In infinite horizon, we resort to ergodic structure of the underlying process. When the running reward is mildly penalising for delaying stopping (i.e., its expectation under the invariant measure is negative), we show that an optimal stopping time exists and is given in a standard form as the time of first entrance to a closed set. This paper generalised Palczewski, Stettner (2014), Stoch Proc Appl 124(12) 3887-3920, by relaxing boundedness of rewards.
\medskip

\end{abstract}

% SIAM format

% \title{Impulse control maximising average cost per unit time: a non-uniformly ergodic case\thanks{Submitted to the editors on 27 June 2016.\funding{Research of both authors has been partly supported by NCN grant DEC-2012/07/B/ST1/03298.}}}
% \author{Jan Palczewski\thanks{School of Mathematics, University of Leeds, LS2 9JT, Leeds, United Kingdom (\email{J.Palczewski@leeds.ac.uk})} \and \L{}ukasz Stettner\thanks{Institute of Mathematics Polish Acad. Sci., Sniadeckich 8, 00-656 Warsaw, Poland, and Vistula University (email{stettner@impan.gov.pl})}}
% \maketitle
% 
% \begin{abstract}
% Of independent interest is an auxiliary result of this paper on undiscounted infinite horizon optimal stopping problem with an unbounded terminal payoff for which continuity of the value function and an optimal strategy are obtained. 
% \end{abstract}
% 
% \begin{keywords}
%  impulse control, ergodic control, optimal stopping, non-uniformly ergodic Markov process, unbounded cost
% \end{keywords}
% 
% \begin{AMS}
% 93E20,  60J25
% \end{AMS}

\section{Introduction}\label{sec:intro}
Let $(X_t)$ be a Feller-Markov process on $(\Omega, F,(F_t))$ with values in a locally compact space $E$ with the metric $\rho$ and Borel $\sigma$ field ${\cal E}$. The process starting from $x$ at time $0$ generates a probability measure $\prob^x$; $\ee^x$ denotes a related expectation operator. Our goal  is to characterise the value function and optimal stopping times of an undiscounted stopping problem
\begin{equation}\label{eqn:main_functional}
v(x) = \sup_\tau \liminf_{T \to \infty} \ee^x \left\{ \int_0^{\tau \wedge T} f(X_s) ds + g(X_{\tau \wedge T}) \right\},
\end{equation}
where $f, g$ are continuous functions.  This is an extension of our results from \cite{palczewski2014}, where we assume that $f$ and $g$ are bounded. Here, we relax the boundedness assumption on $g$ and demonstrate the continuity of the value function and the form of optimal strategies under weak assumptions on the process $(g(X_t))$. Our main assumption is that $\mu(f) < 0$, where $\mu$ is the invariant measure of $(X_t)$, which encourages early stopping similarly as discounting does in a classical case (for details see \cite[Section 2]{palczewski2014}). Apart from being of interest on its own merit, the results of this paper are applied in \cite{Palczewski2016b} where we study impulse control problem with average cost per unit time functional under non-uniform ergodicity of the underlying Markov process and with unbounded costs of interventions. The unbounded costs lead to optimal stopping problems with the terminal reward $g$ that is unbounded from below.

Classical stopping problems for Feller-Markov processes employ discounting:
\[
\sup_\tau \ee^x \left\{ \int_0^{\tau} e^{-rs} f(X_s) ds + e^{-r\tau} g(X_{\tau}) \right\}.
\]
For bounded $f$ and $g$, this ensures that the functional is bounded and, also, that finite horizon problems approximate the above one uniformly in $x$. This implies continuity of the value function and, as a consequence, the form of optimal stopping times. In variational characterisations, the discounting is required to prove existence and uniqueness of solutions, see \cite{bensoussan1984}. Removal of the discounting invalidates all standard approaches. Undiscounted optimal stopping problems for bounded and negative $f, g$ were studied in \cite{Morimoto1991, Morimoto1994, shiryaev1978}. This highlights the aforementioned function of the integral term in \eqref{eqn:main_functional} of penalisation of delaying the stopping decision. It also has a technical advantage of having all terms negative under the expectation. As indicated above, we replace the requirement of $f$ to be negative with the condition that the integral of $f$ with respect to the invariant measure of $(X_t)$ be negative, and we remove boundedness assumptions on $g$.

Another class of methods successfully applied to optimal stopping problems rests on martingale theory and does not require Markovian structure of the underlying process. There, however, explicit integrability assumptions are required: $\ee^x \{ \int_0^\infty f^-(X_s) ds \} < \infty$ and the family of random variables $\{ g^-(X_\tau):\, \text{$\tau$-stopping time} \}$ is $\prob^x$-uniformly integrable, c.f. \cite{Peskir2006}. These assumption do not guarantee that the optimal value is finite, they merely ensure that the Snell envelope is well-defined. In the present paper we allow for both integrals of $f^+$ and $f^-$ to be infinite and do not limit the lower tail of $g(X_t)$. Conversely, we impose integrability assumption of the behaviour of $g^+(X_t)$ and on $f$ to ensure the value function $v(x)$ in \eqref{eqn:main_functional} is finite for any $x\in E$.

% Controlling random systems by impulses, i.e., discrete interventions, is often the only feasible strategy from an application point of view. Therefore, the literature is extensive. For applications in finance, the reader is referred to \cite{DPDS2010, Kelly2011} and references therein. Intensive studies of impulse control of diffusions and diffusions with jumps are presented in \cite{bensoussan1984}. Impulse control of Markov processes with average cost per unit time criterion \eqref{eqn:main_functional} has been studied first in \cite{Robin1981, Robin1983} under uniform ergodicity assumption for constant cost for impulses. These results were extended  to a separated cost (for definition see Proposition \ref{propseparated}) in \cite{Stettner1983a} and to quasicompact transition semigroups in \cite{Stettner1986b}. The problem was also studied under some compactness assumptions in \cite{GatarekStettner}. Ergodic impulse control of diffusion processes was studied in \cite{LionsPerthame1986} and \cite{Perthame1988}. Average cost per unit time functionals have also been widely studied in a different setting where the control affects diffusion process continuously, see the monograph \cite{Arapostathis2012} for a detailed discussion.

The paper is structured as follows. Section 2 provides general assumptions and preliminary results for zero-potentials of centred $f$. Sections 3-4 are devoted to the study of undiscounted stopping problems with unbounded terminal cost for finite and infinite horizon, respectively.

\section{Preliminaries}\label{sec:prem}
For a Markov process $(X_t)$, we define its transition probability measure $P_t(x, \cdot) := \prob^x \{ X_t \in \cdot \}$ and a corresponding semigroup $P_t$, acting on bounded Borel functions, $P_t \phi(x) = \ee^x \{ \phi(X_t) \}$. We make the following assumptions:
\begin{assumption}
\item[(A1)] \label{ass:weak_feller}
(Weak Feller property)
$$
P_t\, \mathcal{C}_0 \subseteq \mathcal{C}_0,
$$
where $\mathcal C_0$ is the space of continuous bounded functions $E \to \er$ vanishing in infinity.
\end{assumption}
\begin{assumption}

\item[(A2)] \label{ass:speed_ergodic} There is a unique probability measure $\mu$ on $\bse$, a function $K: E \to (0, \infty)$ bounded on compacts and a function $h:[0, \infty) \to \er_+$ such that $\int_0^\infty h(t) dt < \infty$ and for any $x \in E$
\[
\| P_t(x, \cdot) - \mu(\cdot) \|_{TV} \leq K(x) h(t),
\]
where $\| \cdot \|_{TV}$ denotes the total variation norm. Furthermore, $\ee^x\left\{K(X_T)\right\}<\infty$ for each $T\geq 0$.
% , and for any compact set $\Gamma\subset E$ and a sequence of sets $A_T \in \ef_T$
% \[
% \lim_{T\to \infty}\sup_{x\in \Gamma} \prob^x\left\{A_T\right\}= 0 \implies \lim_{T\to \infty}\sup_{x\in \Gamma}\ee^x \{\ind{A_T} K(X_T) \}= 0.
% \]
\end{assumption}

Assumption \ref{ass:weak_feller} is necessary to establish the existence of optimal stopping times for general weak Feller processes (a counter-example when it is relaxed is provided at the end of Section 3.1 in \cite{Palczewski2008}). The class of weakly Feller processes \ref{ass:weak_feller} comprises Levy processes \cite[Theorem 3.1.9]{Applebaum2004}, solutions to stochastic differential equations with continuous coefficients driven by Levy processes (see, e.g., \cite[Theorem 6.7.2]{Applebaum2004}). Assumption \ref{ass:speed_ergodic} satisfied by non-uniform geometrically ergodic or polynomially ergodic processes with examples discussed in \cite[Section 6]{palczewski2014}.

% 
% \begin{lemma}
% Under \ref{ass:weak_feller} the operator $P_t$ transforms continuous bounded from above functions into upper semi continuous functions bounded from above.
% \end{lemma}
% \begin{proof}
% By \cite[Corollary 2.2]{Palczewski2008} the semigroup $P_t$ transforms continuous bounded functions into continuous bounded functions. Approximating a continuous function $\varphi$ bounded from above by a sequence of bounded functions $\varphi_n = \max (\varphi,  -n)$ and applying Fatou's lemma completes the proof.
% \end{proof}

Define a centred zero-potential of $f$
\begin{equation}\label{eqalphap}
q(x)=\ee^x\left\{\int_0^\infty \big(f(X_t)-\mu(f)\big)dt\right\}.
\end{equation}
\begin{lemma}\label{lq} \cite[Lemma 2.2]{Palczewski2016b}
Under \ref{ass:weak_feller} and \ref{ass:speed_ergodic}, $q$ is a continuous function and for any bounded stopping time $\tau$
\begin{equation}\label{eq0potent}
q(x)=\ee^x\left\{\int_0^\tau(f(X_t)-\mu(f))dt + q(X_\tau)\right\}.
\end{equation}
% Moreover, for any compact set $\Gamma\subset E$ and a sequence of sets $A_T \in \ef_T$ we have
% \begin{equation}\label{eqn:eq0potent_AT}
% \lim_{T\to \infty}\sup_{x\in \Gamma} \prob^x\left\{A_T\right\}= 0 \implies \lim_{T\to \infty}\sup_{x\in \Gamma}\sup_{\alpha \in [0,1)} \ee^x \{\ind{A_T} |q_\alpha(X_T)| \}= 0.
% \end{equation}
\end{lemma}

\section{Optimal stopping on finite interval with unbounded terminal reward}\label{sec:stopping_T}
In this section we extend classical results on optimal stopping of a bounded functional. We consider the stopping problem
\begin{equation}
w_T(x)=\sup_{\tau \le T} \ee^x\left\{\int_0^\tau f(X_s)ds + g(X_\tau)\right\},
\end{equation}
where $f$ is continuous bounded, and $g$ is continuous but possibly \emph{unbounded}. We will  study continuity of $w_T$ and existence of optimal stopping times by looking at approximations with a sequence of optimal stopping problems
\[
\hat w^n_T(x) = \sup_{\tau \le T} \ee^x\left\{\int_0^\tau f(X_s)ds + (g(X_\tau) \vee (-n)) \wedge n \right\}.
\]
By standard arguments, see e.g. \cite[Corollary 2.3]{Stettner2011}, the mappings $(T, x) \mapsto \hat w^n_T(x)$ are continuous under \ref{ass:weak_feller}.

We will consider the following uniform version of a standard assumption (c.f. \cite[Chapter I, Section 2.2]{Peskir2006})
\begin{assumption}
\item[$(B)_T$] \label{ass:uniform_sup_integrability}
For every $x \in E$ there is a compact ball $K = \bar B(x, \delta)$ such that the random variable $\zeta_T = \sup_{t \in [0, T]} |g(X_t)|$ is uniformly integrable with respect to $\prob^y$ for $y \in K$, i.e.,
\[
\lim_{n \to \infty} \sup_{y \in K} \ee^y \{ \zeta_T \ind{\zeta_T > n} \} = 0.
\]
\end{assumption}
This assumption implies continuity of $w_t(x)$ in $t$ and $x$ as the following lemma shows.

\begin{lemma}\label{lem:continuity_unbounded}
Under  \ref{ass:uniform_sup_integrability}, $\hat w^n_t$ converges to $w_t$ uniformly in $(t, x) \in [0, T] \times K$ for a compact ball $K$ from assumption \ref{ass:uniform_sup_integrability}. If, additionally, assumption \ref{ass:weak_feller} holds, the mapping $(t, x) \mapsto w_t(x)$ is continuous.
\end{lemma}
\begin{proof}
Easily,
\begin{align*}
\big| w_t(x) - \hat w^n_t(x) \big|
&\le \sup_{\tau \le t} \big| \ee^x \{ g(X_\tau) - (g(X_\tau) \vee (-n)) \wedge n \} \big|\\
&\le \sup_{\tau \le t} \ee^x \{ |g(X_\tau)| \ind{|g(X_\tau)| > n} \}
\le \ee^x \{ \zeta_T \ind{\zeta_T > n} \}.
\end{align*}
By assumption \ref{ass:uniform_sup_integrability} the estimate on the right-hand side converges to $0$ when $n \to \infty$ uniformly for $t \in [0, T]$ and $x$ from a compact set $K$. As remarked earlier, under \ref{ass:weak_feller}, $\hat w^n_t(x)$ is continuous in $(t, x)$ and the above uniform convergence implies that so is $w_t(x)$.
\end{proof}

\begin{remark}
In the above proof we could  use  a weaker condition than \ref{ass:uniform_sup_integrability}:
\begin{assumption}\label{ass: weaker general uniform sup}
\item[$(B')_T$]\label{ass: weaker gen} For every $x \in E$ there is a compact ball $K = \bar B(x, \delta)$ such that
\[
\lim_{n \to \infty} \sup_{y \in K} \sup_{\tau \le T} \ee^y \{ |g(X_\tau)| \ind{|g(X_\tau)| > n} \} = 0;
\]
\end{assumption}
or
\begin{assumption}
\item[$(B'')_T$]\label{ass:general}
For every $x \in E$ there is a compact ball $K = \bar B(x, \delta)$ such that
 \begin{equation}\label{eqlab}
    \lim_{R\to \infty}\sup_{y\in K}\sup_{\tau \le T} \ee^y \{\ind{\rho(y,X_{\tau})\geq R} |g(X_{\tau})| \}= 0.
    \end{equation}
\end{assumption}
Notice that \ref{ass: weaker gen} and \ref{ass:weak_feller} imply assumption \ref{ass:general}.
Indeed, for any $n$ and $R>0$ we have 
\begin{align*}
&\sup_{y\in K}\sup_{\tau \le T} \ee^y \{\ind{\rho(y,X_{\tau})\geq R} |g(X_{\tau})| \}\\
&\leq
\sup_{y\in K}\sup_{\tau \le T} \left[\ee^y \{\ind{|g(X_\tau)| > n} |g(X_{\tau})|\}+\ee^y \{\ind{\rho(y,X_{\tau})\geq R}\ind{|g(X_\tau)| \leq n} |g(X_{\tau})|\}\right]\\
&=a(n)+b(n,R).
\end{align*}
For a fixed $n$ we have that $b(n,R)\leq n  \sup_{y\in K}\prob^y \{\exists_{s\in [0,T]}\  \rho(x,X_s)\geq R\}\to 0$ as $R\to \infty$ by assumption \ref{ass:weak_feller} using Proposition 2.1 of \cite{Palczewski2008}. Assumption \ref{ass: weaker gen} implies $a(n)\to 0$.
\end{remark}

We will now prove that there exists an optimal stopping time for $w_T$. Before we formulate the main theorem, we state a simple lemma.
\begin{lemma}\label{lem:integrability_w_T}
\ref{ass:uniform_sup_integrability} implies that $w_{T-t}(X_t)$ is $\prob^x$ integrable.
\end{lemma}
\begin{proof}
\[
\ee^x\left\{|w_{T-t} (X_t)|\right\} \le \|f\| (T-t) + \ee^{x} \{ \zeta_{T} \}.
\]
\end{proof}

\begin{theorem}\label{thm:w_T_stopping}
Under  \ref{ass:weak_feller} and \ref{ass:uniform_sup_integrability}, the smallest optimal stopping time for $w_T(x)$ is given by
\[
\tau_T = \inf \{ s \le T:\ g(X_s) \ge w_{T-s} (X_s) \}.
\]
Moreover, the process $Z^T_t = \int_0^t f(X_s) ds + w_{T-t}(X_t)$ is a right-continuous supermartingale.
\end{theorem}
\begin{proof}
Notice first that from \cite{fakeev1971} it follows that  $Z^{T,n}_t = \int_0^t f(X_s) ds + \hat{w}_{T-t}^n(X_t)$ is a right-continuous supermartingale and by Lemma \ref{lem:continuity_unbounded} and its proof $Z^T_t$ is also a right-continuous supermartingale. By \cite[Theorem 4]{fakeev1970}
\[
\tau_T^\ve = \inf \{ s \le T:\ g(X_s)+\ve \ge w_{T-s} (X_s) \}
\]
is an $\ve$-optimal stopping time and
\begin{equation}\label{eq:epsopt}
w_T(x)=  \ee^x\Big\{\int_0^{\tau_T^\ve} f(X_s)ds + w_{T-\tau_T^\ve}(X_{\tau_T^\ve}) \Big\}.
\end{equation}
Notice that $\tau_T^\ve \leq \tau_T$.
Letting $\ve \to 0$ we have that $\tau_T^\ve$ increases to $\tilde{\tau}$. By quasi left continuity of $(X_t)$, which follows from Theorem 3.13 of \cite{Dynkin1965}, for each positive integer $n$ we have
\[\hat{w}_{T-\tau^\ve}^n(X_{\tau_T^\ve}) \to \hat{w}_{T-\tilde{\tau}}^n(X_{\tilde{\tau}}), \qquad \text{$\prob^x$-a.e.}
\]
as $\ve \to 0$.
Using the arguments of Lemma \ref{lem:continuity_unbounded} we obtain
\[{w}_{T-\tau^\ve}(X_{\tau_T^\ve}) \to {w}_{T-\tilde{\tau}}(X_{\tilde{\tau}}), \qquad \text{$\prob^x$-a.e.} \]
Letting $\ve \to 0$ in \eqref{eq:epsopt} yields
\begin{equation}\label{eqn:12}
w_T(x)
=  \ee^x\left\{\int_0^{\tilde{\tau}} f(X_s)ds + w_{T-\tilde{\tau}}(X_{\tilde{\tau}}) \right\}
 \leq \ee^x\left\{\int_0^{\tilde{\tau}} f(X_s)ds + g_{T-\tilde{\tau}}(X_{\tilde{\tau}})\right\},
\end{equation}
where the inequality is because $\tilde\tau = \lim_{\ve \to 0} \tau^\ve_T$. Hence,  $\tilde{\tau}$ is an optimal stopping time. If $\prob^x \{ \tilde \tau < \tau_T \} > 0$, the inequality in \eqref{eqn:12} would be strict leading to a contradiction. Since $\tilde\tau \le \tau_T$, it follows that $\tilde\tau = \tau_T$, $\prob^x$-a.s.
\end{proof}

Assumption \ref{ass:uniform_sup_integrability} follows from a number of more explicit conditions. Define for any set $U \subseteq E$
\[
\gamma_T(x, U) = \prob^x \{ X_t \in U\ \forall\, t \in [0, T] \}
\]
and
\[
g^*(r) = \sup_{y \in B(0, r)} |g(y)|, \qquad \xi_T = \sup_{t \in [0,T]} \|X_t\|.
\]
\begin{assumption}
\item[$(B1)_T$] \label{ass:point_gamma_prime}
For every $x \in E$ there is a compact ball $K = \bar B(x, \delta)$ and a sequence of compact sets $(K_i)$ such that $K \subset K_1$, $K_i \subset int\, K_{i+1}$, $\bigcup_i K_i = E$ and
\[
\sum_{i=1}^\infty \sup_{y \in K} \Big( \gamma_T(y, K_i) - \gamma_T (y, K_{i-1}) \Big) \max_{y \in K_i} |g(y)| < \infty.
\]

\item[$(B2)_T$] \label{ass:point_gamma prime'}
For every $x \in E$ there is a compact ball $K = \bar B(x, \delta)$ and a sequence of compact sets $(K_i)$ such that $K \subset K_1$, $K_i \subset int\, K_{i+1}$, $\bigcup_i K_i = E$ and
\[
\sum_{i=1}^\infty  \max_{y \in K_{i+1}\setminus K_i} |g(y)| \sup_{y\in K}\prob^y \{ \exists_{t\in [0,T]}; X_t\in K_{i+1}\setminus K_i\} < \infty.
\]

\item[$(B3)_T$] \label{ass:point_max_xi_prime}
For every $x \in E$ there is a compact ball $K = \bar B(x, \delta)$ such that $g^*(\xi_T)$ is uniformly integrable with respect to $\prob^y$, $y \in K$, i.e.
\[
\lim_{N \to \infty} \sup_{y \in K} \ee^y\{ g^*(\xi_T) \ind{\xi_T > N} \} =0.
\]
\end{assumption}
\begin{lemma}
Any of the assumptions \ref{ass:point_gamma_prime}, \ref{ass:point_gamma prime'} or \ref{ass:point_max_xi_prime} is sufficient for \ref{ass:uniform_sup_integrability}.
\end{lemma}
\begin{proof}
Assumption \ref{ass:point_gamma_prime}: Notice that
\begin{align*}
\zeta_T
&\le \sum_{i=1}^\infty \ind{\forall_{t \in [0, T]} X_t \in K_i \text{\ and\ } \exists_{t \in [0, T]} X_t \notin K_{i-1}} \max_{y \in K_i} |g(y)|\\
&= \sum_{i=1}^\infty \big(\ind{\forall_{t \in [0, T]} X_t \in K_i} - \ind{\forall_{t \in [0, T]} X_t \in K_{i-1}}\big) \max_{y \in K_i} |g(y)|.
\end{align*}
Hence, for any $x \in K$
\begin{align*}
\ee^x \{ \zeta_T \ind{\zeta_T > n} \}
&\le \sum_{i=1}^\infty \big(\gamma_T(x, K_i) - \gamma_T(x, K_{i-1})\big) \ind{\max_{y \in K_i} |g(y)| > n} \max_{y \in K_i} |g(y)|\\
&\le \sum_{i=1}^\infty \sup_{y \in K} \big(\gamma_T(y, K_i) - \gamma_T(y, K_{i-1})\big) \ind{\max_{y \in K_i} |g(y)| > n} \max_{y \in K_i} |g(y)|\\
&\to 0 \qquad \text{as $n \to \infty$.}
\end{align*}

Assumption \ref{ass:point_gamma prime'}: For any $x \in K$,
\begin{align*}
\ee^x \{ \zeta_T \ind{\zeta_T > n} \}
&\le \sum_{i=1}^\infty  \ind{\max_{y \in K_{i+1}\setminus K_i} |g(y)| > n} \max_{y \in K_{i+1}\setminus K_i} |g(y)| \prob^x \{ \exists_{t\in [0,T]}; X_t\in K_{i+1}\setminus K_i\}\\
&\le \sum_{i=1}^\infty  \ind{\max_{y \in K_{i+1}\setminus K_i} |g(y)| > n} \max_{y \in K_{i+1}\setminus K_i} |g(y)| \sup_{y \in K} \prob^y \{ \exists_{t\in [0,T]}; X_t\in K_{i+1}\setminus K_i\}\\
&\to 0 \qquad \text{as $n \to \infty$.}
\end{align*}

Assumption \ref{ass:point_max_xi_prime}: It suffices to notice that $\zeta_T \le g^*(\xi_T)$.

\end{proof}

\section{Optimal stopping on infinite interval with unbounded terminal reward}\label{sec:unbounded}
More generally, we consider a stopping problem
\begin{equation}\label{eqn:infinite_stopping}
w(x)=\sup_\tau \liminf_{T \to \infty} \ee^x\left\{\int_0^{\tau \wedge T} f(X_s)ds + g(X_{\tau \wedge T})\right\},
\end{equation}
where $f$ is a continuous bounded function satisfying $\mu(f) := \int_E f(x) \mu(dx) <0$ (recall that $\mu$ is the invariant measure of $(X_t)$), and $g$ is an unbounded continuous function that satisfies
\begin{assumption}
\item[(C1)]\label{ass:upper_bound_g} Random variable $\zeta^+ := \sup_{t \ge 0} g^+(X_t)$ is integrable with respect to $\prob^x$ for any $x$.
\end{assumption}
In the supremum above, we allow for stopping times taking the value infinity. Notice, however, that \eqref{eqn:infinite_stopping} is equivalent to
\begin{equation}\label{eqn:infinite_stopping_1}
w(x)=\sup_{\text{$\tau$-bounded}} \ee^x\left\{\int_0^{\tau} f(X_s)ds + g(X_{\tau})\right\},
\end{equation}
since for any stopping time $\tau$ we have
\[
\liminf_{T \to \infty} \ee^x\left\{\int_0^{\tau \wedge T} f(X_s)ds + g(X_{\tau \wedge T})\right\} \le \sup_{\text{$\sigma$-bounded}} \ee^x\left\{\int_0^{\sigma} f(X_s)ds + g(X_{\sigma})\right\}.
\]
Similar arguments show that
\begin{equation}\label{eqn:infinite_stopping_2}
w(x)=\sup_\tau \limsup_{T \to \infty} \ee^x\left\{\int_0^{\tau \wedge T} f(X_s)ds + g(X_{\tau \wedge T})\right\}.
\end{equation}
We use the formulation in \eqref{eqn:infinite_stopping} instead of \eqref{eqn:infinite_stopping_1} as it allows for a simple description of $\ve$-optimal and optimal stopping times as hitting times of a compact set, in line with the classical theory of optimal stopping.

Two main results in this section are Theorem \ref{thm:tau_star} which shows the form of optimal stopping times for $w(x)$ and Theorem \ref{prop:continuous_w1} in which the continuity of $w$ is established. We also state conditions under which $\liminf$ in \eqref{eqn:infinite_stopping} can be omitted.

\begin{remark}
Assumption \ref{ass:upper_bound_g} is  weaker than often made in the optimal stopping literature, where one requires $\ee^x \{ \sup_{t \ge 0} |g(X_t)| \} < \infty$, c.f. \cite[Section 1.2]{Peskir2006}.
\end{remark}

 The following assumption allows us to omit the limit in \eqref{eqn:infinite_stopping} for integrable stopping times:
\begin{assumption}
\item[(C2)]\label{ass:g_vanish} For any sequence of events $A_T \in \ef_T$, $T > 0$,
\[
\lim_{T \to \infty} \prob^x(A_T) = 0 \quad \Longrightarrow \quad \lim_{T \to \infty} \ee^x \{ \ind{A_T} g^-(X_T) \} = 0.
\]
\end{assumption}
\begin{lemma}\label{lem:integrable_equiv}
Under assumption \ref{ass:g_vanish}, for any integrable stopping time $\tau$ the following equality holds
\[
\lim_{T \to \infty} \ee^x\left\{\int_0^{\tau \wedge T} f(X_s)ds + g(X_{\tau \wedge T})\right\}
=
\ee^x\left\{\int_0^{\tau} f(X_s)ds + g(X_{\tau})\right\}.
\]
\end{lemma}
\begin{proof}
Let $\ee^x \{\tau \} < \infty$. Due to the boundedness of $f$, the convergence of the integral terms is obvious. For the terminal reward, we have
\begin{align*}
\ee^x \{ g(X_\tau) - g(X_{\tau \wedge T}) \}
&= \ee^x \{ g^+(X_\tau) - g^+(X_{\tau \wedge T}) \} + \ee^x \{ g^-(X_\tau) - g^-(X_{\tau \wedge T}) \} \\
&\le \ee^x \{ g^+(X_\tau) - g^+(X_{\tau \wedge T}) \} + 0 - \ee^x \{ \ind{\tau > T} g^-(X_T) \}.
\end{align*}
The first term converges to $0$ as $T \to \infty$ by dominated convergence theorem. The last term vanishes by assumption \ref{ass:g_vanish}. This proves  $\liminf_{T \to \infty} \ee^x \{ g(X_{\tau \wedge T}) \} \ge \ee^x \{ g(X_\tau)\}$. The opposite inequality follows by Fatou's lemma (recall that $g(X_t)$ is bounded from above by an integrable random variable $\zeta^+$): \(\limsup_{T \to \infty} \ee^x \{ g(X_{\tau \wedge T}) \} \allowbreak \le \ee^x \{ g(X_\tau)\}\).
\end{proof}
\begin{remark}
Without assumption \ref{ass:g_vanish}, if a stopping time $\tau$ has a finite expectation then we can only show that
\begin{equation}\label{eqn:integrable_ineq}
\begin{aligned}
&\liminf_{T \to \infty} \ee^x\left\{\int_0^{\tau \wedge T} f(X_s)ds + g(X_{\tau \wedge T})\right\}\\
&\le
\limsup_{T \to \infty} \ee^x\left\{\int_0^{\tau \wedge T} f(X_s)ds + g(X_{\tau \wedge T})\right\}
\le
\ee^x\left\{\int_0^{\tau} f(X_s)ds + g(X_{\tau})\right\},
\end{aligned}
\end{equation}
by applying Fatou's lemma to the convergence of $g(X_{\tau \wedge T})$ to $g(X_\tau)$.
\end{remark}

The convergence in Lemma \ref{lem:integrable_equiv} can be made uniform over compact sets under a uniform version of assumptions \ref{ass:upper_bound_g} and \ref{ass:g_vanish}:
\begin{assumption}
\item[(C1')]\label{ass:unif_upper_bound_g} Random variable $\zeta^+ := \sup_{t \ge 0} g^+(X_t)$ is uniformly integrable with respect $\prob^x$ for $x$ from compact sets.
\item[(C2')]\label{ass:g_vanish_prime} For any compact set $\Gamma \subset E$ and a sequence of events $A_T \in \ef_T$, $T > 0$,
\[
\lim_{T \to \infty} \sup_{x \in \Gamma} \prob^x(A_T) = 0 \quad \Longrightarrow \quad \lim_{T \to \infty} \sup_{x \in \Gamma} \ee^x \{ \ind{A_T} |g(X_T)| \} = 0.
\]
\end{assumption}
\begin{cor}\label{cor:uniintegrable equiv}
Under \ref{ass:unif_upper_bound_g} and \ref{ass:g_vanish_prime}, if $x \mapsto \ee^x\left\{\tau\right\}$ is bounded on a compact set $\Gamma$, the convergence of $\ee^x\left\{\int_0^{\tau \wedge T} f(X_s)ds + g(X_{\tau \wedge T})\right\}$ to $\ee^x\left\{\int_0^{\tau} f(X_s)ds + g(X_{\tau})\right\}$ is uniform on $\Gamma$.
\end{cor}

For the reminder of this section we make the following standing assumptions: \ref{ass:weak_feller}, \ref{ass:speed_ergodic}, \ref{ass:uniform_sup_integrability} for $T > 0$, \ref{ass:upper_bound_g} and \ref{ass:g_vanish}. We also assume:
\begin{assumption}
\item[(C3)]\label{ass:undis_large_dev} For any $x \in E$, there is $d(x) < 0$ such that
\[
\gamma (x) = \sup_{\tau} \liminf_{T \to \infty} \ee^x \Big\{ \int_0^{\tau \wedge T} \big(f(X_s) - d(x)\big) ds \Big\} < \infty.
\]
\end{assumption}
This will ensure that the value function $w(x)$ is finite and allow us to prove that $\ve$-optimal stopping times for \eqref{eqn:infinite_stopping} are integrable. Interested reader is referred to \cite[Sections 2.2-2.3]{palczewski2014} for a thorough discussion of sufficient conditions for \ref{ass:undis_large_dev}. Here we only mention a condition that links $d(x)$ with $\mu(f)$.
\begin{lemma}
A sufficient condition for \ref{ass:undis_large_dev} can be formulated as
\begin{assumption}
\item[(S)]\label{ass:suffundis} $\mu(f)<0$ and for some $\delta\in (0,1]$
\[\bar{\gamma}(x)=\sup_{\tau} \liminf_{T \to \infty} \ee^x \Big\{(1-\delta)\mu(f)(\tau \wedge T)-q(X_{\tau \wedge T}) \Big\} < \infty,
\]
\end{assumption}
where $q$ is defined in Section \ref{sec:prem}. Then \ref{ass:undis_large_dev} holds with $d(x)=\delta \mu(f)$.
\end{lemma}
\begin{proof}
It is sufficient to notice that using Lemma \ref{lq} for any bounded stopping time $\sigma$ we have
\begin{align*}
\ee^x \Big\{(1-\delta)\mu(f)\sigma-q(X_{\sigma}) \Big\}
&=
\ee^x \Big\{(1-\delta)\mu(f)\sigma + \int_0^\sigma \big(f(X_s) - \mu(f)\big) ds \Big\} - q(x)\\
&=
\ee^x \Big\{\int_0^\sigma \big(f(X_s) - \delta \mu(f)\big) ds \Big\} - q(x).
\end{align*}
\end{proof}
\begin{remark}\label{rem:suffundis}
Notice that assumption \ref{ass:suffundis} is satisfied, in particular, when the negative part of $q$ is bounded.
\end{remark}

Now we prove that the value function $w$ takes finite values and is lower-semi\-con\-ti\-nu\-ous which will allow us to define candidates for $\ve$-optimal and optimal stopping times as hitting times of closed sets.
\begin{lemma}\label{lem:lsc_w}
Function $w$ is finite and lower semi-continuous.
\end{lemma}
\begin{proof}
By \ref{ass:undis_large_dev}, recalling that $d(x) < 0$, we have for any bounded stopping time $\tau$:
\begin{align*}
\ee^x \left\{ \int_0^\tau f(X_s) ds + g(X_\tau) \right\} = \ee^x \left\{ \int_0^\tau \big( f(X_s) - d(x) \big) ds + d(x) \tau + g(X_\tau) \right\} \le \gamma(x) + \ee^x \{ \zeta^+ \}.
\end{align*}
Stopping problem \eqref{eqn:infinite_stopping} is equivalent to optimising over the set of all bounded stopping times, c.f. \eqref{eqn:infinite_stopping_1}. Hence $w(x) \le \gamma(x) + \ee^x\{ \zeta^+ \} < \infty$. Also, $w_T(x)$ converges to $w(x)$  from below for any fixed $x$. Functions $w_T$ are continuous (Lemma \ref{lem:continuity_unbounded}), hence $w$ is lower semi-continuous.
\end{proof}

Define
\begin{align*}
\tau_\ve &= \inf \{ t \ge 0:\ w(X_t) \le g(X_t) + \ve \}, \\
\tau^* &= \inf \{ t \ge 0:\ w(X_t) \le g(X_t)  \}.
\end{align*}
These are stopping times due to the lower semi-continuity of $w-g$ and the right-continuity of the process $X_t$. In the following theorem we prove their optimality.

\begin{theorem}\label{thm:tau_star} Under the assumptions  \ref{ass:weak_feller}, \ref{ass:speed_ergodic}, \ref{ass:upper_bound_g}-\ref{ass:undis_large_dev} and \ref{ass:uniform_sup_integrability} satisfied for each $T>0$,
the stopping time $\tau^*$ is optimal for $w$ and $\ee^x \{ \tau^* \} \le Z(x)$, where
\begin{equation}\label{eqn:Mx}
Z(x) = \frac{\gamma(x) + \ee^x\{\zeta^+\}- g(x) + 1}{-d(x)}.
\end{equation}
Moreover, $\tau_\ve$ is $\ve$-optimal and $\tau^* = \lim_{\ve \to 0} \tau_\ve$.
\end{theorem}

The proof will be preceded by auxiliary lemmas in which we will use the assumptions listed above without stating them explicitely.

\begin{lemma}\label{lem:bound_for_stop}
For every  $\ve$-optimal stopping time $\sigma$ we have
\[
\ee^x \{ \sigma \} \le \frac{\gamma(x) + \ee^x\{\zeta^+\}-w(x) + \ve}{-d(x)},
\]
where $\zeta^+$ was defined in assumption \ref{ass:upper_bound_g}.
\end{lemma}
\begin{proof}
For any stopping time $\sigma $ and $T>0$  we have
\begin{align*}
\ee^x \Big\{ \int_0^{\sigma \wedge T} f(X_s)ds + g(X_{\sigma \wedge T})\Big\}
&=
\ee^x \Big\{ \int_0^{\sigma\wedge T} \big(f(X_s)-d(x)\big)ds + d(x) \big(\sigma \wedge T \big) + g(X_{\sigma \wedge T})\Big\}\\
&\le
\gamma(x) + \ee^x \{\zeta^+ \} + d(x) \ee^x\{ \sigma \wedge T\}.
\end{align*}
Therefore
\begin{equation}\label{estimation}
-d(x)\ee^x \big\{ \sigma \wedge T \big\}\le \gamma(x) + \ee^x \{\zeta^+ \}  - \ee^x \Big\{ \int_0^{\sigma\wedge T} f(X_s)ds + g(X_{\sigma\wedge T})\Big\}.
\end{equation}
If $\sigma$ is $\ve$-optimal then \[\liminf_{T\to \infty}\ee^x \Big\{ \int_0^{\sigma\wedge T} f(X_s)ds + g(X_{\sigma\wedge T})\Big\}\geq w(x)-\ve\]
and letting $\limsup_{T\to \infty}$ in \eqref{estimation} we complete the proof.
\end{proof}

The above lemma implies that the expectation of every $\ve$-optimal stopping time $\sigma$ with $\ve < 1$ is bounded by $Z(x)$ defined in \eqref{eqn:Mx}.

\begin{lemma}\label{lem:undisc_sigma_m}
For every $x \in E$, there exists a non-decreasing sequence $\sigma_m$ of bounded $\frac1m$-optimal stopping times for $w(x)$.
\end{lemma}
\begin{proof}
Functions $w_T(x)$ approximate $w(x)$ from below. By Theorem \ref{thm:w_T_stopping}, stopping problems $w_T(x)$ admit optimal solutions of the form
\[
\tau_T = \inf \{ t \in [0, T]:\ g(X_t) \ge w_{T-t}(X_t) \}.
\]
As $w_T$ are non-decreasing in $T$, $\tau_T$ are non-decreasing in $T$. Taking $T(m) = \inf \{ T \ge 0:\ w(x) - w_T(x) \le \frac1m \}$, we can set $\sigma_m = \tau_{T(m)}$.
\end{proof}

% The proof of the following lemma is reproduced for the convenience of the reader.
% \begin{lemma}\label{lem:bound_expectation} (\cite[Lemma 2.10]{palczewski2014})
% Let $(X_n)$ be a sequence of positive random variables with $\ee X_n \le 1$. If $Y$ is a positive random variable (with a possibly infinite expectation) and $\lim_{n \to \infty} \prob\{Y > X_n\} =0$ then $\ee Y \le 1$.
% \end{lemma}
% \begin{proof}
% Assume first that $Y$ is bounded by $K > 0$ a.s. Then
% \(
% \ee \{ Y\} = \ee \{ \ind{Y \le X_n} Y \} + \ee \{ \ind{Y > X_n} Y \} \le \ee \{ \ind{Y \le X_n} X_n \} + \ee \{ \ind{Y > X_n} Y \}
% \le \ee \{ X_n \} + K \prob \{ Y > X_n\} \le 1 + K \prob \{ Y > X_n\}
% \)
% and the last term tends to $0$ as $n \to \infty$. Take now an arbitrary $Y$ satisfying assumptions of the lemma. For each $K > 0$, we have $\lim_{n \to \infty} \prob\{Y\wedge K > X_n\} = 0$ and $\ee \{ Y \wedge K \} \le 1$. By monotone convergence, $\ee \{ Y \} = \lim_{K \to \infty} \ee \{ Y \wedge K \} \le 1$.
% \end{proof}

\begin{lemma}\label{lem:super_mart}
The process $Z_t := \int_0^t f(X_s) ds + w(X_t)$ is a right-continuous $\prob^x$-supermartingale for any $x \in E$. Moreover, for a bounded stopping time $\sigma$ and an arbitrary stopping time $\tau$
\begin{equation}\label{eqn:undisc_bellman_ineq}
\ee^{x} \Big \{ \int_0^\sigma f(X_s) ds + g(X_\sigma) \Big\}
\le \ee^{x} \Big\{ \int_0^{\tau \wedge \sigma} f(X_s) ds + \ind{\sigma < \tau} g(X_{\sigma}) + \ind{\sigma \ge \tau} w(X_{\tau}) \Big\}.
\end{equation}
\end{lemma}
\begin{proof}
By Theorem \ref{thm:w_T_stopping}, the process $Z_t^T := \int_0^t f(X_s) ds + w_{T-t}(X_t)$, $t \in [0, T]$, is a right-continuous $\prob^x$-supermartingale. Therefore, $w_T(x) \ge \ee^x \{ \int_0^t f(X_s) ds + w_{T-t}(X_t) \}$ and so $\ee^x \{ w_{T-t} (X_t) \} \le w_T(x) + \|f\| t$. Since $w_{T-t}(x)$ is increasing in $T$, the monotone convergence theorem implies $\ee^x \{w(X_t)\} \le w(x) + \|f\| t$ and $w(X_t) \in L^1(\prob^x)$.

Notice that $Z^T_t$ and $w_{T-t}$ are increasing in $T$, so $\ee^x \{ Z^T_t \} \le \|f\| t + \ee^x \{ w(X_t) \}$. Hence $Z_t = \sup_{T \ge t} Z_t^T$ is a $\prob^x$-integrable process which is right-continuous by \cite[Theorem T16, Chapter VI]{Meyer1966}. We will show that $Z_t$ equals $\tl Z_t = \int_0^t f(X_s) ds + w(X_t)$ and is a supermartingale. Since $\tl Z_t - Z^T_t = w(X_t) - w_{T-t}(X_t)$, by monotone convergence theorem $Z^T_t$ converges to $\tl Z_t$ in $L^1(\prob^x)$ and due to monotonicity also pointwise, hence $\tl Z_t = Z_t$. This also proves that the supermartingale property of $(Z^T_t)$ is transferred to $(Z_t)$.

Since $\sigma$ is a bounded stopping time, the optional sampling theorem yields $\ee^x \{ Z_{\sigma} | \ef_{\tau \wedge \sigma} \} \le Z_{\tau \wedge \sigma}$. This reads
\(
\int_0^{\tau \wedge \sigma} f(X_s) ds + w(X_{\tau \wedge \sigma}) \ge \ee^{x} \Big \{ \int_0^\sigma f(X_s) ds + w(X_\sigma) \Big| \ef_{\tau \wedge \sigma} \Big\}.
\)
Hence,
\begin{align*}
\int_0^{\tau \wedge \sigma} f(X_s) ds + \ind{\sigma \ge \tau} w(X_{\tau \wedge \sigma})
&\ge
\ee^{x} \Big \{ \int_0^\sigma f(X_s) ds + \ind{\sigma \ge \tau} w(X_\sigma) \Big| \ef_{\tau \wedge \sigma} \Big\}\\
&\ge
\ee^{x} \Big \{ \int_0^\sigma f(X_s) ds + \ind{\sigma \ge \tau} g(X_\sigma) \Big| \ef_{\tau \wedge \sigma} \Big\}.
\end{align*}
Adding $\ind{\sigma < \tau} g(X_\sigma)$ to both sides completes the proof.
\end{proof}

\begin{proof}[Proof of Theorem \ref{thm:tau_star}]
Let $\sigma_m$ be a sequence of $\frac1m$-optimal stopping times. 
Apply \eqref{eqn:undisc_bellman_ineq} to $\sigma_m$ and $\tau_\ve$ and notice that $w(X_{\tau_\ve}) \le g(X_{\tau_\ve}) + \ve$ due to the continuity of $g$, lower semicontinuity of $w$ and right-continuity of the process $(X_t)$:
\begin{align*}
w(x) - \frac1m
&\le \ee^{x} \Big\{ \int_0^{\sigma_m\wedge \tau_\ve} f(X_s) ds + \ind{\sigma_m < \tau_\ve} g(X_{\sigma_m}) + \ind{\sigma_m \ge \tau_\ve} w(X_{\tau_\ve}) \Big\}\\
&\le \ee^{x} \Big\{ \int_0^{\sigma_m\wedge \tau_\ve} f(X_s) ds + \ind{\sigma_m < \tau_\ve} g(X_{\sigma_m}) + \ind{\sigma_m \ge \tau_\ve} \big(g(X_{\tau_\ve}) + \ve \big)\Big\}\\
&\le \ee^{x} \Big\{ \int_0^{\sigma_m\wedge \tau_\ve} f(X_s) ds + g(X_{\sigma_m \wedge \tau_\ve}) \Big\} + \ve.
\end{align*}
Since by Lemma \ref{lem:bound_for_stop} the stopping time $\tau_\ve$ is $\prob^x$-integrable, dominated convergence theorem implies $\lim_{m \to \infty} \ee^{x} \big\{ \int_0^{\sigma_m\wedge \tau_\ve} f(X_s) ds \big\} = \ee^{x} \big\{ \int_0^{\tau_\ve} f(X_s) ds \big\}$. Recalling that $g(X_{\sigma_m \wedge \tau_\ve})$ is bounded from above by an integrable random variable $\zeta^+$, Fatou's lemma yields $\limsup_{m \to \infty} \ee^{x} \{g(X_{\sigma_m \wedge \tau_\ve}) \} \le \ee^{x} \{ g(X_{\tau_\ve})\}$, where we used continuity of $g$ and quasi left-continuity of $(X_t)$. Combining these results gives
\(
w(x) \le \ee^{x} \Big\{ \int_0^{\tau_\ve} f(X_s) ds + g(X_{\tau_\ve}) + \ve \Big\},
\)
so, using Lemma \ref{lem:integrable_equiv}, $\tau_\ve$ is $\ve$-optimal.

Stopping times $\tau_\ve$ are increasing in $\ve$ with the expectation bounded by $Z(x)$, so $\tau_0 = \lim_{\ve \to 0} \tau_\ve$ is well-defined and $\ee^x \{ \tau_0 \} \le Z(x)$. For any $0 < \ve \le \eta$ we have $g(X_{\tau_\ve}) \ge w(X_{\tau_\ve}) - \eta$. Using the quasi left-continuity of $(X_t)$ and lower semicontinuity of $w$, we take the limit $\ve \to 0$ and then $\eta \to 0$ to obtain $g(X_{\tau_0}) \ge w(X_{\tau_0})$. So $\tau_0 \ge \tau^*$. This means that $\tau^*$ is finite and $\ee^x \{ \tau^* \} \le Z(x)$. Its optimality follows in the same way as $\ve$-optimality of $\tau_\ve$.
\end{proof}

For the continuity of $w$, we need a version of assumption \ref{ass:undis_large_dev} which is uniform over compact sets 
\begin{assumption}
\item[(C3')]\label{ass:unif_undis_large_dev} There is a function $d:E \to (-\infty, 0)$ such that for any compact set $\Gamma \subset E$ we have
$\sup_{x\in \Gamma} d(x)<0$ and
\[
\sup_{x\in \Gamma}\gamma (x) =\sup_{x\in \Gamma} \sup_{\tau} \liminf_{T \to \infty} \ee^x \Big\{ \int_0^{\tau \wedge T} \big(f(X_s) - d(x)\big) ds \Big\} < \infty.
\]
\end{assumption}

\begin{prop} \label{prop:continuous_w}
Under the assumptions of Theorem \ref{thm:tau_star} and \ref{ass:unif_upper_bound_g}-\ref{ass:unif_undis_large_dev}, the function $w$ is continuous.
\end{prop}
\begin{proof}
Let $\tau^*$ be the optimal stopping time for $w$ from Theorem \ref{thm:tau_star}. Let
\[
v_T(x) = \ee^x \Big\{ \int_0^{\tau^* \wedge T} f(X_s) ds + g(X_{\tau^* \wedge T}) \Big\}.
\]
Then $v_T(x) \le w_T(x) \le w(x)$ for all $x$ and $T$. By Corollary \ref{cor:uniintegrable equiv}, $v_T$ converges to $w$ uniformly on compact sets as the expectation of $\tau^*$ is bounded on compact sets by assumption \ref{ass:unif_undis_large_dev}, see the definition of $Z(x)$. Hence, $w_T$ converges to $w$ uniformly on compact sets. Continuity of $w_T$ (Lemma \ref{lem:continuity_unbounded}) implies then the continuity of $w$.
\end{proof}

\begin{remark}
When $g$ is bounded from above then assumption \ref{ass:unif_upper_bound_g} holds trivially with $\zeta^+ = \| g^+\|$.
\end{remark}
\begin{remark}
When $q$ is bounded from below and $\mu(f) < 0$ then by \cite[Lemma 2.16]{palczewski2014} the assumption \ref{ass:unif_undis_large_dev} holds with any $d(x)\in [\mu(f),0)$ and $\gamma(x)\leq q(x)-\|q^-\|$.
\end{remark}

Although assumption \ref{ass:unif_undis_large_dev} is not particularly restrictive, it is not clear how to verify it for specific examples. The following theorem relaxes it and proves the continuity of $w$ under the assumption that $\mu(f) < 0$.

\begin{theorem} \label{prop:continuous_w1}
Assume \ref{ass:weak_feller}-\ref{ass:speed_ergodic}, \ref{ass:uniform_sup_integrability} satisfied for each $T>0$, and \ref{ass:unif_upper_bound_g}-\ref{ass:g_vanish_prime}. If $\mu(f)<0$, function $w$ is continuous and $\tau^*$ is an optimal stopping time.
\end{theorem}
\begin{proof}
This proof is based on ideas from the proof of Theorem 2.28 in \cite{palczewski2014}.

Assume first that the set $\{ x \in E:\ f(x) \le \mu(f) \}$ is compact. By \cite[Lemma 2.19]{palczewski2014} the zero-potential $q$ is bounded from below and Lemma 2.18 and 2.16 in \cite{palczewski2014} imply that assumption \ref{ass:unif_undis_large_dev} holds.

 Take now an arbitrary continuous bounded function $f$ with $\mu(f) < 0$. Denote by $B_n$ a ball with radius $n$ and some fixed centre independent of $n$. Let $z_n(x) = 1 - \rho(x, B_n) \wedge 1$, where $\rho(x, B_n)$ is the distance of $x$ from the ball $B_n$. Take $N$ large enough so that $\int_{B_N^c} |f(x)| \mu(dx) < -\mu(f)/4$, where $B_N^c$ is the complement of $B_N$. Define $\hat f(x) = z_N(x) f(x)$ and $\bar f = f \vee \hat f$. Then
\[
\mu(\bar f) \le \int_{B_N} f(x) \mu(dx) + \int_{B^c_N} |f(x)| \mu(dx) \le \mu(f) - \mu(f) / 4 - \mu(f) / 4 = \mu(f) / 2.
\]
Hence, $\mu(\bar f) < 0$. Moreover, $\bar f (x) \ge 0$ for $x \in B_{N+1}^c$, so the set $\{x\in E:\ \bar f(x) \le \mu(\bar f) \}$ is contained in $B_{N+1}$ and compact. The function $\bar f$ satisfies the conditions in the first part of the proof and, hence, assumption \ref{ass:unif_undis_large_dev}. However, since $\bar f \ge f$, then function $f$ satisfies assumption \ref{ass:unif_undis_large_dev} with the same $d(x)$. It suffices now to apply Proposition \ref{prop:continuous_w} to prove the continuity of $w$ and Theorem \ref{thm:tau_star} to obtain that the stopping time $\tau^*$ is optimal.
\end{proof}

\bibliographystyle{amsplain}

\bibliography{references-undis}

\end{document}